\documentclass{amsart}
\usepackage{amsmath,amssymb,amscd, amsthm,amsfonts}
\usepackage{graphicx}
\usepackage[hidelinks]{hyperref}

\newtheorem{theorem}{Theorem}[section]
\newtheorem{lemma}[theorem]{Lemma}

\newtheorem{proposition}[theorem]{Proposition}

\theoremstyle{remark}
\newtheorem{remark}[theorem]{Remark}
\newtheorem*{acknow}{Acknowledgments}

\newcommand{\abs}[1]{\lvert#1\rvert}

\begin{document}

\title[Overdetermined problems in space forms]{Overdetermined problems for fully nonlinear equations with constant Dirichlet boundary conditions in space forms }

\author{Shanze Gao}
\address{School of Mathematics and Statistics, Shaanxi Normal University, Xi’an, 710119, P. R. China}
\email{gaoshanze@snnu.edu.cn}

\author{Hui Ma}
\author{Mingxuan Yang}
\address{Department of Mathematical Sciences, Tsinghua University, Beijing, 100084, P. R. China}
\email{ma-h@mail.tsinghua.edu.cn}
\email{ymx20@mails.tsinghua.edu.cn}

\begin{abstract}
We consider overdetermined problems for two classes of fully nonlinear equations with constant Dirichlet boundary conditions in a bounded domain in space forms. We prove that if the domain is star-shaped, then the solution to the Hessian quotient overdetermined problem is radially symmetric. By establishing a Rellich-Poho\v{z}aev type identity for the $k$-Hessian equation with constant Dirichlet boundary condition, we also show the radial symmetry of the solution to the $k$-Hessian overdetermined problem for some boundary value without star-shapedness assumption of the domain.
\end{abstract}



\maketitle

\section{Introduction}\label{sec1}

Let $(M^{n}(K), g)$ be a space form of constant sectional curvature $K$ with metric $g$. In this paper, we consider the following class of overdetermined problems:
\begin{equation}\label{Eq:hessian}
\left\{
\begin{aligned}
&\dfrac{\sigma_{k}(\nabla^{2}u+Kug)}{\sigma_{l}(\nabla^{2}u+Kug)}=\dfrac{\binom{n}{k}}{\binom{n}{l}} &\text{in }&~\overline{\Omega}, \\
&u=Kc_{1}&\text{on }&~\partial \Omega, \\
&u_{\nu}=c_{2}&\text{on }&~\partial \Omega, \\
&Ku\geq 0 &\text{in }&~\Omega,
\end{aligned}
\right.
\end{equation}
where $0\leq l<k\leq n$, $\Omega$ is a bounded, open, connected domain in $M^{n}(K)$ with a boundary $\partial \Omega$ of class $C^2$, $c_{1}$ and $c_{2}>0$ are constants, $\nu$ denotes the outward unit normal to $\partial{\Omega}$ and $\sigma_k(\nabla^2 u+Kug)$ is the $k$-th elementary symmetric function of the eigenvalues of $\nabla^2 u+Kug$ (see Section \ref{Sec:symmetricfunction}).

In a pioneer work \cite{serrin}, Serrin proved the symmetry of solutions to overdetermined problems for various elliptic differential equations in $\mathbb{R}^n$  via the moving planes method. In particular, he proved that, the solution to the Poisson equation \[ \Delta u=-1 \quad \text{in}~\Omega\subset \mathbb{R}^{n} \] with boundary conditions \[ u=0,\quad u_{\nu}=\mathrm{constant}\quad \text{on}~\partial\Omega \]  is radially symmetric and that $\Omega$ is a ball.
In a subsequent paper \cite{Weinberger}, Weinberger presented a short proof of the same result by applying the maximum principle to an auxiliary function, which is often referred to as $P$-function. Alternative proofs inspired by Weinberger's approach have been obtained in \cite[Theorem 2]{BNST}, \cite[Theorem 2.1]{MP}, and \cite[Problem I]{Payne-Schaefer}. The problems for the other operators, such as $p$-Laplacian and anisotropic $p$-Laplacian in $\mathbb{R}^{n}$, interested readers may refer to \cite{G.P, BC, Cianchi-Salani, Garofalo-Lewis, Wang-Xia, C. Xia-J. Yin, C-Li} etc.

In space forms, Molzon \cite{Molzon} used a $P$-function to prove symmetry results for the equation $\Delta u=V$ (see the definition of $V$ in Section \ref{Sec:prelim}) in $\mathbb{S}^{n}_{+}$ and used the moving planes method to obtain the symmetry results for the equation $\Delta u=-1$. Kumaresan and Prajapat \cite{K-P} got the symmetry results for the equation $\Delta u+g(u)=0$ in space forms, also by the moving planes method, where $g$ is a $C^{1}$ function, under the condition $u>0$. 
In \cite{Qiu-Xia}, Qiu and Xia used two auxiliary functions, $P$ and $\widetilde{P}$, 
to obtain the symmetry of solution for Poisson equation $\Delta u+nu=n$ in the sphere $\mathbb{S}^{n}$. Serrin-type overdetermined problems for the $k$-Hessian equation in $\mathbb{R}^{n}$ are considered by Brandolini, Nitsch, Salani and Trombetti in \cite{BNST}. Recently, Gao, Jia and Yan \cite{Gao-Jia-Yan} applied the maximum principle to a $P$-function to prove the symmetry results for the $k$-Hessian equation in hyperbolic space $\mathbb{H}^{n}$ and reproved the Theorem 2 in \cite{BNST}. For more papers about the problems in space forms, interested readers may refer to \cite{Ciraolo-Vezzoni17, Ciraolo-Vezzoni19, AF-AR} and references therein.

In this paper, we suppose that $\Omega$ is a bounded, open, connected domain in $M^{n}(K)$ with a boundary $\partial\Omega$ of class $C^{2}$. For the case $K>0$, we additionally suppose that $M^n(K)$ is the hemisphere $S^n_+(\frac{1}{\sqrt{K}})$ with radius $\frac{1}{\sqrt{K}}$.

Our first result concerns the overdetermined problem for the Hessian quotient equation.
\begin{theorem}\label{Th:QuoHess}
For given $0\leq l<k\leq n$, if $\Omega$ is star-shaped and there exists a solution $u \in C^{3}(\Omega)\cap C^{2}(\overline{\Omega})$ to problem \eqref{Eq:hessian}, then $u$ is radially symmetric and $\Omega$ is a geodesic ball in $M^n(K)$.
\end{theorem}

\begin{remark} For the case $K=0$, the condition $Ku\geq 0$ in $\Omega$ holds automatically and the constant Dirichlet boundary condition reduces to the zero one $u\vert_{\partial\Omega}=0$. Moreover, the condition may be replaced by $u\vert_{\partial\Omega}=c$ for any constant $c\in\mathbb{R}$. In fact,
the equation now becomes 
\[ \dfrac{\sigma_{k}(\nabla^{2}u)}{\sigma_{l}(\nabla^{2}u)}=\dfrac{\binom{n}{k}}{\binom{n}{l}}, \] 
thus $\bar{u}=u-c$ satisfies the same equation 
and the zero boundary condition. 

In the two other cases, the conditions of the problem (\ref{Eq:hessian}) actually imply a restriction on the constant $c_{1}$.
\item[(1)]
For the case $K>0$, the conditions $u_{\nu}\vert_{\partial\Omega}=c_{2}>0$ and $Ku\geq 0$ in $\Omega$ imply $u\vert_{\partial\Omega}=Kc_{1}>0$. Thus, $c_{1}>0$.
 \item[(2)]
For the case $K<0$, the condition $Ku\geq 0$ in $\Omega$ implies $Ku\vert_{\partial\Omega}\geq 0$ by the continuity of $u$. Substituting the boundary condition $u\vert_{\partial\Omega}=Kc_{1}$, we get $K^{2}c_{1}\geq 0$ and therefore $c_{1}\geq 0$.
\end{remark}

\begin{remark}\label{model}
The radially symmetric solution $u$ in Theorem \ref{Th:QuoHess} are given  explicitly as follows:
\item[(i)] For $K=0$, \[u=\dfrac{\vert x\vert^{2}-c_{2}^{2}}{2};\]
\item[(ii)] For $K>0$, \[u=\dfrac{1}{K}-\dfrac{c_{2}}{\sqrt{K}\sin(\sqrt{K}R) }\cos(\sqrt{K}r),\] where $R=\dfrac{1}{\sqrt{K}}\arctan\left(\dfrac{c_{2}\sqrt{K}}{1-K^{2}c_{1}}\right)$;
\item[(iii)] For $K<0$, \[u=\dfrac{1}{K}+\dfrac{c_{2}}{\sqrt{-K}\sinh(\sqrt{-K} R)}\cosh(\sqrt{-K}r),\] where $R=\dfrac{1}{\sqrt{-K}}\text{arctanh} \left(\dfrac{c_{2}\sqrt{-K}}{1-K^{2}c_{1}}\right)$.

Note that inequality $K^{2}c_{1}< 1$ is needed in (ii) and (iii), but not assumed in Theorem \ref{Th:QuoHess} (compared with Theorem \ref{Th:Hess}). Actually, the inequality will be confirmed in the proof of Theorem $\ref{Th:QuoHess}$.
\end{remark}

To prove Theorem \ref{Th:QuoHess}, we use two auxiliary functions $P$ and $\widetilde{P}$ (see Section \ref{Sec:P}). The latter is a generalization of Qiu-Xia's auxiliary function $\widetilde{P}$ in \cite{Qiu-Xia} for space form cases. By the maximum principle and the Hopf lemma, we see that either $P$ or $\widetilde{P}$ is constant, which leads to our conclusion. Notice that the condition $Ku\geq 0$ in $\Omega$ enables us to establish the $k$-admissible property of the solution $u$ (see Lemma \ref{Th:binGammak}) and apply the maximum principle to $P$-function (see Lemma \ref{Th:P}).  
For the case $K>0$, the condition that $\Omega$ lies in a hemisphere  $S^n_+(\frac{1}{\sqrt{K}})$  guarantees the positivity of the function $V$ (see Section \ref{Space forms}) 
which is important in Lemma \ref{Th:tildeP}, while the condition is not necessary for \cite[Theorem 1.2]{Qiu-Xia}. In fact, the corresponding term in \cite[Lemma 2.2]{Qiu-Xia} vanishes due to the linearity of the Laplace operator.

We then turn to a study of the $k$-Hessian equations with boundary conditions $u\vert_{\partial\Omega}=Kc_{1}$ and $u_{\nu}\vert_{\partial\Omega}=c_{2}$ in space forms.
Following the original scheme of Weinberger's argument, 
we establish a Rellich-Poho\v{z}aev type identity. Using this identity, we obtain our second result.

\begin{theorem}\label{Th:Hess}
Suppose that $u \in C^{3}(\Omega)\cap C^{2}(\overline{\Omega})$ is a solution to the problem
\begin{equation}\label{Eq:}
\left\{  
\begin{aligned} 
&\sigma_{k}(\nabla^{2}u+Kug)=\binom{n}{k} &\text{in }&~\Omega,\\  
&u=Kc_{1}&\text{on }&~\partial \Omega,\\
&u_{\nu}=c_{2}&\text{on }&~\partial \Omega,\\
&K u\geq 0 &\text{in }&~\Omega,
\end{aligned}
\right.
\end{equation} 
where $c_{2}$ is a positive constant. In addition, the constant $c_{1}$ satisfies the condition $K^{2}c_{1}\leq 1$ when $2\leq k\leq n$, while $c_{1}$ can be arbitrary when $k=1$. Then $u$ is radially symmetric and $\Omega$ is a geodesic ball in $M^n(K)$.
\end{theorem}

\begin{remark}
Note that the assumption $K^{2}c_{1}\leq 1$ in the above theorem replaces the star-shapedness of  $\Omega$  in Theorem \ref{Th:QuoHess}. For the case $K=0$, conditions $K^{2}c_{1}\leq 1$ and $Ku \geq 0$ hold automatically, so the theorem reduces to the results in \cite{BNST}. For the case $K<0$ and $c_{1}=0$, inequality $K^{2}c_{1}\leq 1$ also holds automatically and $Ku\geq 0 \text{~in }\Omega$ is implied by the maximum principle, so the result reduces to Theorem 1.2 in \cite{Gao-Jia-Yan}. For the special case $k=1$, as in Section \ref{Sec:ProofkHess}, the condition $K^{2}c_{1}\leq 1$ is not necessary.

\end{remark}

As far as we know, the overdetermined problems with constant (possibly nonzero) Dirichlet boundary conditions have not received sufficient attention yet.
We propose the constant boundary condition $u=Kc_{1}$ on $\partial \Omega$, which is compatible with the condition $Ku\geq 0$ in $\Omega$. 
A nonzero constant $c_{1}$ makes many differences. The model solutions (see Remark \ref{model}), values of $P$-function and $\widetilde{P}$-function, the Rellich-Poho\v{z}aev type identity and many arguments in proofs are all related to constant $c_{1}$. See the later sections for details.

The paper is organized as follows. In Section \ref{Sec:prelim}, we recall some notations and facts of space forms, elementary symmetric functions and Hessian operators. In Section \ref{Sec:convexity}, we show the $k$-convexity of $u$ to problem (\ref{Eq:hessian}). In Section \ref{Sec:P}, we consider $P$-function and $\widetilde{P}$-function to the Hessian quotient equations. Then we prove Theorem \ref{Th:QuoHess} in Section \ref{Sec:ProofQuoHess}. In Section \ref{Sec:R-Pidentity}, we establish a Rellich-Poho\v{z}aev type identity for the $k$-Hessian equations with constant boundary conditions in space forms. In the last section, we prove Theorem \ref{Th:Hess}.

\section{Notation and preliminaries}
\label{Sec:prelim}

\subsection{Space forms}\label{Space forms}

Let  $M^{n}(K)$ be a complete, simply connected manifold with constant sectional curvature $K$. 
It is well-known that 
$M^{n}(K)$ is isometric to the Euclidean space if $K=0$, the hyperbolic space if $K<0$ and the sphere if $K>0$. These models can be described as the warped product manifold $M^{n}(K)=[0,\bar{r})\times \mathbb{S}^{n-1}$ with metric \[ g=dr\otimes dr+f(r)^{2}g_{\mathbb{S}^{n-1}}, \] where $r$ is the geodesic distance from any point $x$ to a given point $x_{0}$ in $M^{n}(K)$ and $g_{\mathbb{S}^{n-1}}$ is the metric of the $(n-1)$-dimensional standard unit sphere. The warping function $f(r)$ is given by
\begin{equation*}
f(r)=
\begin{cases}
r \qquad &\text{for }K=0,\\
\dfrac{\sinh(\sqrt{-K}r)}{\sqrt{-K}}\qquad &\text{for }K<0,\\
\dfrac{\sin(\sqrt{K}r)}{\sqrt{K}}  \qquad &\text{for }K>0.
\end{cases}
\end{equation*}
In addition, we require that $\bar{r}=\infty$ for $K\leq 0$ and $\bar{r}= \frac{\pi}{2\sqrt{K}}$ for $K>0$. 
Hence our restriction on $\bar{r}$ implies we only focus on a hemisphere $S^n_+(\frac{1}{\sqrt{K}})$ with radius $\frac{1}{\sqrt{K}}$ in the latter case.

We know that space forms $M^n(K)$ endow a natural conformal Killing vector field, $f(r)\frac{\partial}{\partial r}$, in terms of the above warped product model. It is the gradient of the potential function $\Phi$ defined by
\begin{equation*}
\Phi(r):=
\begin{cases}
\dfrac{1}{2}r^{2}, \qquad &\text{for } K=0,\\
\dfrac{\cosh(\sqrt{-K}r)}{-K}, \qquad &\text{for }K<0,\\
\dfrac{-\cos(\sqrt{K}r)}{K},  \qquad &\text{for }K>0,
\end{cases}
\end{equation*}

Denote 
\begin{equation*}
V(r):= f^{\prime}(r)=
\begin{cases}
1, \qquad &\text{for } K=0,\\
\cosh(\sqrt{-K}r), \qquad &\text{for } K<0,\\
\cos(\sqrt{K}r), \qquad &\text{for } K>0.
\end{cases}
\end{equation*}
Thus for any vector field $\xi$ on $M^{n}(K)$,  
\[ \nabla_{\xi}\left( f(r)\frac{\partial}{\partial r} \right)=V\xi. \]
Then by direct computation we have the following nice properties.
\begin{proposition}\label{Th:PhiV}
$-K\nabla \Phi=\nabla V$, $\nabla^{2} \Phi =Vg$ and $\nabla^{2} V=-KVg$.
\end{proposition}

Let $\Omega$ be a bounded, open, connected domain in $M^{n}(K)$. 
With the assumption $\bar{r}= \frac{\pi}{2\sqrt{K}}$ in the case $K>0$ (i.e., $\Omega\subset S^n_+(\frac{1}{\sqrt{K}})$ with radius $\frac{1}{\sqrt{K}}$), we know $V(r)>0$ in $\Omega$ for any constant $K$.

Let $\nu$ denote the outward unit normal of $\partial\Omega$. We say $\Omega$ is \emph{star-shaped}, if $g(\nu,\frac{\partial}{\partial r})(y)>0$ for all $y\in \partial\Omega$.

\subsection{Elementary symmetric functions}\label{Sec:symmetricfunction}

We recall some properties of elementary symmetric polynomials which will be used later.

  For  $k\in \{1,\ldots, n\}$, the $k$-th elementary symmetric function of $\lambda=(\lambda_1,\cdots,\lambda_n)\in \mathbb{R}^n$  is defined by 
\[ 
\sigma_{k}(\lambda):=\sum_{1\leq i_{1}<\cdots<i_{k}\leq n}\lambda_{i_{1}}\cdots \lambda_{i_{k}}. 
\] 

Given a real symmetric $n\times n$ matrix $A=(a_{ij})$ with eigenvalues $\lambda(A)$, 
we can define the $k$-th elementary symmetric polynomial by
$\sigma_{k}(A):=\sigma_{k}(\lambda(A))$. Thus  
\[ \sigma_{k}(A)=\frac{1}{k!}
\sum_{\substack{1\leq i_{1},\ldots, i_{k} \leq n \\ 
1\leq j_{1}, \ldots, j_{k} \leq n}}
\delta^{j_{1}\cdots j_{k}}_{i_{1}\cdots i_{k}}a_{i_{1}j_{1}}\cdots a_{i_{k}j_{k}}, \] 
where $\delta^{j_{1}\cdots j_{k}}_{i_{1}\cdots i_{k}}$ is the generalized Kronecker symbol defined by
\begin{equation*}
\delta^{j_{1}\cdots j_{k}}_{i_{1}\cdots i_{k}}=\begin{cases}
1,\quad &\text{if }(i_{1}\cdots i_{k})\text{ is an even permutation of }(j_{1}\cdots j_{k}), \\
-1,\quad &\text{if }(i_{1}\cdots i_{k})\text{ is an odd permutation of }(j_{1}\cdots j_{k}), \\
0,\quad &\text{otherwise}.
\end{cases}
\end{equation*}
We also set $\sigma_{0}(A)=1$ and $\sigma_{k}(A)=0$ for $k<0$ and $k>n$.

Denote $\sigma_{k}^{ij}(A):=\frac{\partial\sigma_{k}(A)}{\partial a_{ij}}$. The following properties (see Proposition 1.2 in \cite{Reilly73} for reference) are useful in later calculations.
\begin{proposition}\label{Th:BasicSigmak}
\begin{flalign*}
&(\romannumeral1)\sum_{i,j=1}^{n}\sigma^{ij}_{k}(A)a_{ij}=k\sigma_{k}(A),\\
&(\romannumeral2)\sum_{i=1}^{n}\sigma_{k}^{ii}(A)=(n-k+1)\sigma_{k-1}(A),\\
&(\romannumeral3)~\sigma_{k}^{ij}(A)=\sigma_{k-1}(A)\delta_{ij}-\sum_{l=1}^{n}\sigma_{k-1}^{il}(A)a_{jl},\\
&(\romannumeral4)\sum_{i,j,l=1}^{n}\sigma^{il}_{k}(A)a_{jl}a_{ij}=\sigma_{1}(A)\sigma_{k}(A)-(k+1)\sigma_{k+1}(A).&
\end{flalign*}
\end{proposition}

Now we collect some inequalities related to the elementary symmetric functions with references for the proofs. 

For $1\leq k\leq n$, we recall that the G{\aa}rding's cone is defined by
\[ \Gamma_{k}:=\{ \lambda\in \mathbb{R}^{n}\vert\sigma_{i}(\lambda)>0,\text{ for all }1\leq i\leq k \}. \] 
We say that a real symmetric matrix $A$ lies in  $\Gamma_{k}$ if its eigenvalues $\lambda(A)\in \Gamma_{k}$.

\begin{proposition}[Theorem 15.18, \cite{lieb}] \label{Th:positivedefinite}
For $A=(a_{ij})\in \Gamma_{k}$ and $0\leq l<k\leq n$, the matrix $\left(\frac{\partial}{\partial a_{ij}}\left(\frac{\sigma_{k}(A)}{\sigma_{l}(A)}\right)\right)$ is positive definite.
\end{proposition}

We recall the Newton-MacLaurin inequalities. 
\begin{proposition}[Theorems 51 and 52 in \cite{Hardy}]
For $1\leq k\leq n-1$ and $A=(a_{ij})\in \Gamma_{k}$,
\begin{equation}\label{ineq:N-M}
\frac{\frac{\sigma_{k+1}(A)}{\binom{n}{k+1}}}{\frac{\sigma_{k}(A)}{\binom{n}{k}}} \leq \frac{\frac{\sigma_{k}(A)}{\binom{n}{k}}}{\frac{\sigma_{k-1}(A)}{\binom{n}{k-1}}}
\end{equation}
and
\begin{equation*}
\frac{\sigma_{k+1}(A)}{\binom{n}{k+1}}\leq \left(\frac{\sigma_{k}(A)}{\binom{n}{k}}\right)^{\frac{k+1}{k}}.
\end{equation*}
For both inequalities, the equality occurs if and only if $A=cI$ for some $c>0$, where $I$ is the identity matrix.
\end{proposition}

Since we are considering the Hessian quotient equations, we need some variants of the above inequality.
\begin{lemma}\label{Th:IneqLem}
If 
\begin{equation} \label{eq:sigma_k/l}
\frac{\sigma_{k}(A)}{\sigma_{l}(A)}=\frac{\binom{n}{k}}{\binom{n}{l}} 
\end{equation}
and $A\in \Gamma_{k}$ for $0\leq l< k\leq n$, then
\begin{equation*}
\frac{\sigma_{k-1}(A)}{\sigma_{k}(A)}\geq \frac{k}{n-k+1}~\text{ and }~\frac{\sigma_{l+1}(A)}{\sigma_{l}(A)}\geq \frac{n-l}{l+1}.
\end{equation*}
Each equality occurs if and only if $A=cI$ for some $c>0$, where $I$ is the identity matrix. 
Moreover, we also have 
\begin{equation*}
\frac{\sigma_{k+1}(A)}{\sigma_{k}(A)}\leq \frac{n-k}{k+1}~\text{ and }~\frac{\sigma_{l-1}(A)}{\sigma_{l}(A)}\leq \frac{l}{n-l+1}.
\end{equation*}
\end{lemma}

\begin{proof}
Notice
\begin{equation*}
\frac{\frac{\sigma_{k}(A)}{\binom{n}{k}}}{\frac{\sigma_{l}(A)}{\binom{n}{l}}}=\frac{\frac{\sigma_{k}(A)}{\binom{n}{k}}}{\frac{\sigma_{k-1}(A)}{\binom{n}{k-1}}}\frac{\frac{\sigma_{k-1}(A)}{\binom{n}{k-1}}}{\frac{\sigma_{k-2}(A)}{\binom{n}{k-2}}}\cdots\frac{\frac{\sigma_{l+2}(A)}{\binom{n}{l+2}}}{\frac{\sigma_{l+1}(A)}{\binom{n}{l+1}}}\frac{\frac{\sigma_{l+1}(A)}{\binom{n}{l+1}}}{\frac{\sigma_{l}(A)}{\binom{n}{l}}}.
\end{equation*}
It follows from the Newton-MacLauring inequality \eqref{ineq:N-M} that
\begin{equation*}
\left( \frac{\frac{\sigma_{k}(A)}{\binom{n}{k}}}{\frac{\sigma_{k-1}(A)}{\binom{n}{k-1}}} \right)^{k-l}\leq \frac{\frac{\sigma_{k}(A)}{\binom{n}{k}}}{\frac{\sigma_{l}(A)}{\binom{n}{l}}} \leq \left( \frac{\frac{\sigma_{l+1}(A)}{\binom{n}{l+1}}}{\frac{\sigma_{l}(A)}{\binom{n}{l}}} \right)^{k-l}.
\end{equation*}
Then, 
the condition \eqref{eq:sigma_k/l} implies 
\begin{equation*}
\frac{\frac{\sigma_{k}(A)}{\binom{n}{k}}}{\frac{\sigma_{k-1}(A)}{\binom{n}{k-1}}}\leq1 \leq \frac{\frac{\sigma_{l+1}(A)}{\binom{n}{l+1}}}{\frac{\sigma_{l}(A)}{\binom{n}{l}}}.
\end{equation*}
Thus it leads to
\[ \frac{\sigma_{k}(A)}{\sigma_{k-1}(A)}\leq \frac{n-k+1}{k}, \quad\quad \text{and}\quad   \frac{\sigma_{l+1}(A)}{\sigma_{l}(A)}\geq \frac{n-l}{l+1}. \]
Furthermore, if $\sigma_{k+1}(A)>0$ or $l\geq 1$, then the Newton-MacLaurin inequality implies \[ \frac{\sigma_{k+1}(A)}{\sigma_{k}(A)}\leq \frac{n-k}{k+1} \] or \[ \frac{\sigma_{l-1}(A)}{\sigma_{l}(A)}\leq \frac{l}{n-l+1}. \]
If $\sigma_{k+1}(A)\leq 0$ or $l=0$, then the above inequalities hold automatically.
\end{proof}

\subsection{Hessian operators}

Let $\Omega$ be an open domain in $M^n(K)$ and $u\in C^2(\Omega)$. The $k$-Hessian operator $\sigma_k(\nabla^2 u+Kug)$ is defined as the $k$-th elementary symmetric function of $\nabla^2u+Kug$, where $\nabla^{2}u$ denotes the Hessian of $u$ and $g$ is the metric of $M^{n}(K)$. 

For $2\leq k\leq n$, the $k$-th Hessian operator is fully nonlinear. We are also interested in the Hessian quotient operator $\frac{\sigma_{k}(\nabla^{2}u+Kug)}{\sigma_{l}(\nabla^{2}u+Kug)}$,
for $0\leq l<k\leq n$, 
which is a more general class of fully nonlinear operators. 
When $l=0$, it is a $k$-Hessian operator.

From now on, we calculate under a local orthonormal frame $\{ e_{1}, \cdots,e_{n} \}$ and denote a symmetric 2-tensor $b:=\nabla^{2}u+Kug$ and $\nabla_{k}b_{ij}:=b_{ijk}$. Einstein's summation convention is used for repeated indexes unless otherwise stated. 

\begin{proposition}
The 2-tensor $b$ is a Codazzi tensor and the $k$-Hessian operator is divergence-free. In other words, $b_{ijk}=b_{ikj}$ and
for any $u\in C^{3}(\Omega)$, we have
\[ \sum_{i}\nabla_{i}\big(\sigma_{k}^{ij}(\nabla^{2}u+Kug)\big)=0, \]
with respect to any local orthornormal frame $\{e_1, \cdots, e_n\}$. 
\end{proposition}

\begin{proof}
Recall that the Riemmanian curvature tensor of space form $M^{n}(K)$ is given by $R_{ijkl}=K(\delta_{ik}\delta_{jl}-\delta_{il}\delta_{jk})$. By the Ricci identity, we have 
\begin{align*}
b_{ijk}&=u_{ijk}+Ku_{k}\delta_{ij} \\
&=u_{ikj}+u_{m}R_{mijk}+Ku_{k}\delta_{ij} \\
&=u_{ikj}+Ku_{m}(\delta_{mj}\delta_{ik}-\delta_{mk}\delta_{ij})+Ku_{k}\delta_{ij} \\
&=u_{ikj}+Ku_{j}\delta_{ik}=b_{ikj}.
\end{align*}
This means $b$ is a Codazzi tensor. For such $b=\nabla^{2}u+Kug$, a direct computation (refer to Proposition 2.1 in \cite{Reilly73} or Proposition 2.3 in \cite{pietra}) yields that the $k$-Hessian operator is divergence-free.
\end{proof}

\section{$k$-convexity of $u$}
\label{Sec:convexity}

Now, we show that the solution $u$ to the overdetermined problem \eqref{Eq:hessian} is naturally $k$-admissible, i.e., $\nabla^{2}u+Kug\in \Gamma_{k}$. This property is used to ensure that we can use the maximum principle and the Newton-MacLaurin inequality later.

\begin{lemma}\label{Th:binGammak}
For $2\leq k\leq n$ and $0\leq l<k$, if $\Omega$ is a $C^{2}$ bounded domain and $u\in C^{2}(\overline{\Omega})$ is a solution of \eqref{Eq:hessian}, then $(\nabla^{2}u+Kug)\in \Gamma_{k}$ for all $x\in \overline{\Omega}$.
\end{lemma}

\begin{proof}
The proof for the case $l=0$ can be found in \cite{BNST} and \cite{Gao-Jia-Yan}. Hence we assume $1\leq l<k$. From $u_{\nu}=c_{2}>0$ on $\partial\Omega$, we know there is a point $x_{0}\in \Omega$ such that $\displaystyle u(x_{0})=\min_{\overline{\Omega}}u$. Then $\nabla^{2}u(x_{0})\geq 0$. From assumption $Ku\geq 0$ in $\Omega$, we know $(\nabla^{2}u+Kug)(x_{0})\geq 0$. Moreover, \[ \frac{\sigma_{k}(\nabla^{2}u+Kug)}{\sigma_{l}(\nabla^{2}u+Kug)}=\dfrac{\binom{n}{k}}{\binom{n}{l}}>0 \] implies $\sigma_{k}(\nabla^{2}u+Kug)(x_{0})>0$.

Next, we show $\sigma_{k}(\nabla^{2}u+Kug)(x)>0$ for all $x\in \overline{\Omega}$. Obviously, $\sigma_{k}(\nabla^{2}u+Kug)=0$ cannot occur since \[ \frac{\sigma_{k}(\nabla^{2}u+Kug)}{\sigma_{l}(\nabla^{2}u+Kug)}>0. \] If there exists $y_{0}\in \overline{\Omega}$ such that $\sigma_{k}(\nabla^{2}u+Kug)(y_{0})<0$, from $\sigma_{k}(\nabla^{2}u+Kug)(x_{0})>0$ and the smoothness of $u$, we know there exists $z_{0}\in \overline{\Omega}$ such that $\sigma_{k}(\nabla^{2}u+Kug)(z_{0})=0$ which is impossible. Hence $\sigma_{k}(\nabla^{2}u+Kug)(x)>0$ for all $x\in\overline{\Omega}$.

The boundary condition implies $\vert\nabla u\vert=c_{2}>0$ on $\partial\Omega$ which means $\nabla u\neq 0$ on $\partial\Omega$. By the implicit function theorem, we know $\partial\Omega$ is a hypersurface in $M^{n}(K)$ and $\nu=\frac{\nabla u}{\vert\nabla u\vert}$. By a suitable choice of frame such that $e_{n}=\nu$, Hessian of $u$ at any fixed point of $\partial\Omega$ has the following form:
\begin{equation*}
\nabla^{2}u=\begin{pmatrix}
c_{2}\kappa_{1} & 0 & 0 &\cdots &0 & u_{1n} \\
0 & c_{2}\kappa_{2} & 0 &\cdots &0 & u_{2n} \\
0 & 0 & c_{2}\kappa_{3} &\cdots &0 & u_{3n} \\
\vdots & \vdots & \vdots &\ddots &\vdots & \vdots \\
0 & 0 & 0 & \cdots & c_{2}\kappa_{n-1} & u_{n-1\,n} \\
u_{n1} & u_{n2} & u_{n3} & \cdots &u_{n\,n-1} & u_{nn}
\end{pmatrix},
\end{equation*}
where $\kappa_{1},...,\kappa_{n-1}$ are principal curvatures of $\partial\Omega$. On $\partial\Omega$, $u_{n}=c_{2}$ implies $u_{1n}=\cdots =u_{n-1\,n}=0$. Thus, matrix $\nabla^{2}u$ is diagonal and
\begin{equation*}
\nabla^{2}u+Kug=\begin{pmatrix}
c_{2}\kappa_{1}+Ku &   & &  \\
  & \ddots &  & \\
  &  & c_{2}\kappa_{n-1}+Ku &  \\
 &  &   & u_{nn}+Ku
\end{pmatrix}.
\end{equation*}
Its $(n-1)\times (n-1)$ submatrix is denoted by
\begin{equation*}
c_{2}\kappa+KuI:=\begin{pmatrix}
c_{2}\kappa_{1}+Ku &   &  \\
  & \ddots &   \\
  &  & c_{2}\kappa_{n-1}+Ku
\end{pmatrix}.
\end{equation*}
Then
\begin{equation}\label{Eq:sigmak>0}
0<\sigma_{k}(\nabla^{2}u+Kug)=(u_{nn}+Ku)\sigma_{k-1}(c_{2}\kappa+KuI)+\sigma_{k}(c_{2}\kappa+KuI).
\end{equation}

It is known that the boundary  $\partial\Omega$ has at least one elliptic point, i.e., there exists at least one point $y\in \partial\Omega$ such that all principal curvatures $\kappa_{1}(y),...,\kappa_{n-1}(y)$ are nonnegative. Combining with $c_{2}>0$, $Ku\geq 0$ and inequality \eqref{Eq:sigmak>0}, we know $\sigma_{k-1}(c_{2}\kappa+KuI)(y)>0$. Then from the Newton-MacLaurin inequality, we have
\begin{align*}
\sigma_{k-1}(\nabla^{2}u+Kug)(y)&=(u_{nn}+Ku)\sigma_{k-2}(c_{2}\kappa+KuI)+\sigma_{k-1}(c_{2}\kappa+KuI) \\
&> -\frac{\sigma_{k-2}(c_{2}\kappa+KuI)\sigma_{k}(c_{2}\kappa+KuI)}{\sigma_{k-1}(c_{2}\kappa+KuI)}+\sigma_{k-1}(c_{2}\kappa+KuI) \\
&>0.
\end{align*}
Similarly, we can get $\sigma_{j}(\nabla^{2}u+Kug)(y)>0$ for all $1\leq j\leq k$. Therefore, $\nabla^{2}u+Kug\in \Gamma_{k}$ at $y$.

Now, we consider the set $S:=\{ x\in\overline{\Omega}\vert\nabla^{2}u+Kug\in\Gamma_{k} \}$. The above argument shows $S$ is nonempty. From $u\in C^{2}(\overline{\Omega})$, we know $S$ is a relatively open set in $\overline{\Omega}$. For any $\bar{x}$ in the relative closure $\overline{S}$ in $\overline{\Omega}$,  there exists a sequence $\{ x_{i} \}_{i=1}^{\infty}$ in $S$ such that  $\lim_{i\rightarrow \infty} x_{i}=\bar{x}$. Then $(\nabla^{2}u+Kug)(\bar{x})\in\overline{\Gamma}_{k}$. From the Newton-MacLaurin inequality and $\sigma_{k}(\nabla^{2}u+Kug)(\bar{x})>0$, we know that $(\nabla^{2}u+Kug)(\bar{x})\in\Gamma_{k}$. This means  $\bar{x}\in S$, therefore  $S$ is also relatively closed in $\overline{\Omega}$. From the connectedness of $\overline{\Omega}$, we know that $S=\overline{\Omega}$. Hence, $(\nabla^{2}u+Kug)\in \Gamma_{k}$ for all $x\in \overline{\Omega}$.

\end{proof}

\section{$P$-function and $\widetilde{P}$-function}
\label{Sec:P}

First, we consider an auxiliary function $P:=\abs{\nabla u}^{2}+Ku^{2}-2u$. It is a generalization of Weinberger's function in \cite{Weinberger} and has been studied in \cite{Qiu-Xia,Ciraolo-Vezzoni19} for linear equations in space forms. For fully nonlinear equations, Ma \cite{X.N.Ma} used the $P$-function for the $2$-dimensional Monge-Ampere equation to obtain a necessary condition of solvability for the capillarity boundary problem. The $P$-functions for the $k$-Hessian equations in the Euclidean space and in the hyperbolic space are studied in \cite{PS} and \cite{Gao-Jia-Yan}.

Let $F(b):=\frac{\sigma_{k}(b)}{\sigma_{l}(b)}$ and $F^{ij}:=\frac{\partial F(b)}{\partial b_{ij}}$ for $b=\nabla^2 u+Kug$. Proposition \ref{Th:positivedefinite} and Lemma \ref{Th:binGammak} show that the matrix $\left(F^{ij}\right)$ is positive definite. Hence the operator $F^{ij}\nabla^{2}_{ij}$ is elliptic.

\begin{lemma}\label{Th:P} 
If $u\in C^{3}(\Omega)$ be a solution to problem \eqref{Eq:hessian} in Theorem \ref{Th:QuoHess}, then \[ F^{ij}\nabla^{2}_{ij}P\geq 0. \] Moreover, either \[P=K^{3}c_{1}^{2}+c_{2}^{2}-2Kc_{1}\quad \text{in }~\overline{\Omega} \] or \[ P<K^{3}c_{1}^{2}+c_{2}^{2}-2Kc_{1}\quad \text{in }~\Omega. \]
\end{lemma}

\begin{proof}
By direct computation, we have $$\frac{1}{2}\nabla_{i}(\vert\nabla u\vert^{2}+Ku^2)=u_{mi}u_{m}+Kuu_{i}=b_{mi}u_{m},$$ and 
\begin{align*}
\frac{1}{2}\nabla^{2}_{ij}P&=b_{mij}u_{m}+b_{mi}u_{mj}-u_{ij} \\
&=b_{mij}u_{m}+b_{mi}b_{mj}-Kub_{ij}-b_{ij}+Ku\delta_{ij}.
\end{align*}
From Proposition \ref{Th:BasicSigmak}, we know
\begin{align*}
F^{ij}b_{ij}&=F\left( \frac{\sigma_{k}^{ij}b_{ij}}{\sigma_{k}}-\frac{\sigma_{l}^{ij}b_{ij}}{\sigma_{l}} \right)=(k-l)F, \\
F^{ij}\delta_{ij}&=F\left( \frac{\sigma_{k}^{ij}\delta_{ij}}{\sigma_{k}}-\frac{\sigma_{l}^{ij}\delta_{ij}}{\sigma_{l}} \right)=F\left( (n-k+1)\frac{\sigma_{k-1}}{\sigma_{k}}-(n-l+1)\frac{\sigma_{l-1}}{\sigma_{l}} \right), 
\end{align*}
and
\begin{align*}
F^{ij}b_{mi}b_{mj}&=F\left( \frac{\sigma_{k}^{ij}b_{mi}b_{mj}}{\sigma_{k}}-\frac{\sigma_{l}^{ij}b_{mi}b_{mj}}{\sigma_{l}} \right) \\
&=F\left( \sigma_{1}-(k+1)\frac{\sigma_{k+1}}{\sigma_{k}}-\sigma_{1}+\frac{(l+1)\sigma_{l+1}}{\sigma_{l}} \right) \\
&=F\left(-(k+1)\frac{\sigma_{k+1}}{\sigma_{k}}+\frac{(l+1)\sigma_{l+1}}{\sigma_{l}} \right).
\end{align*}
We also know $\nabla_{m}F=0$ since $F$ is constant.
Then
\begin{align*}
\frac{1}{2}F^{ij}\nabla^{2}_{ij}P&=u_{m}\nabla_{m}F+F^{ij}b_{mi}b_{mj}-KuF^{ij}b_{ij}-F^{ij}b_{ij}+KuF^{ij}\delta_{ij} \\
&=F\bigg( -(k+1)\frac{\sigma_{k+1}}{\sigma_{k}}+\frac{(l+1)\sigma_{l+1}}{\sigma_{l}}-(k-l)Ku-(k-l) \\
&\hspace{3em}+(n-k+1)Ku\frac{\sigma_{k-1}}{\sigma_{k}}-(n-l+1)Ku\frac{\sigma_{l-1}}{\sigma_{l}} \bigg) \\
&=F\bigg( -(k+1)\frac{\sigma_{k+1}}{\sigma_{k}}+\frac{(l+1)\sigma_{l+1}}{\sigma_{l}}-(k-l)\bigg) \\
&\quad +KuF\bigg(-(k-l)+(n-k+1)\frac{\sigma_{k-1}}{\sigma_{k}}-(n-l+1)\frac{\sigma_{l-1}}{\sigma_{l}} \bigg).
\end{align*}
From Lemma \ref{Th:IneqLem} and $Ku\geq 0$, we obtain $F^{ij}\nabla^{2}_{ij}P\geq 0$. 

Since $F^{ij}\nabla^{2}_{ij}$ is elliptic, by the strong maximum principle, if there exists $x_{0}\in \Omega$ such that $P(x_{0})=\max_{\overline{\Omega}}P$, then $P$ is constant. From boundary conditions, we know \[P=K^{3}c_{1}^{2}+c_{2}^{2}-2Kc_{1}\quad \text{in }~\overline{\Omega}. \] If $P(x)<P\vert_{\partial\Omega}$ for all $x\in \Omega$, then \[ P<K^{3}c_{1}^{2}+c_{2}^{2}-2Kc_{1}\quad \text{in }~\Omega. \]
\end{proof}

Now we introduce another auxiliary function
\begin{equation*}
\widetilde{P}:=
-g(\nabla u,\nabla \Phi)+u V+\Phi.
\end{equation*}
It is worth noting that in the case $K=1$ the function $\widetilde{P}$ can be written as $$\widetilde{P}=g(\nabla u,\nabla V)+uV-V,$$ which is the same as the second auxiliary function in \cite{Qiu-Xia}. From {Lemma 2.2} in \cite{Qiu-Xia}, it follows that  $\Delta\widetilde{P}=0$ for the solution $u$ to  $\Delta u+nKu=n$. For the Hessian quotient equations, we obtain the following lemma.

\begin{lemma}\label{Th:tildeP}
Let $u\in C^{3}(\Omega)$ be a solution to problem \eqref{Eq:hessian} in Theorem \ref{Th:QuoHess}, then \[ F^{ij}\nabla^{2}_{ij}\widetilde{P}\geq 0. \]
\end{lemma}

\begin{proof}
Using Proposition \ref{Th:PhiV}, by direct computation, we have
\begin{align*}
\nabla_{i}\widetilde{P}&=-u_{mi}\Phi_{m}-u_{m}\Phi_{mi}+V_{i}u+Vu_{i}+\Phi_{i}\\
&=-u_{mi}\Phi_{m}+V_{i}u+\Phi_{i}
\end{align*}
and
\begin{align*}
\nabla^{2}_{ij}\widetilde{P}&=-u_{mij}\Phi_{m}-u_{mj}\Phi_{mi}+V_{ij}u+V_{i}u_{j}+\Phi_{ij}\\
&=-u_{mij}\Phi_{m}-u_{mj}\delta_{mi}V-KVu\delta_{ij}-K\Phi_{i}u_{j}+V\delta_{ij} \\
&=-(u_{mi}+Ku\delta_{mi})_{j}\Phi_{m}-V(u_{ij}+Ku\delta_{ij})+V\delta_{ij}\\
&=-b_{mij}\Phi_{m}-Vb_{ij}+V\delta_{ij}.
\end{align*}
We also know $\nabla_{m}F=0$ since $F$ is constant.
Combining Proposition \ref{Th:BasicSigmak}, we obtain
\begin{align*}
F^{ij}\nabla^{2}_{ij}\widetilde{P}&=-\Phi_{m}\nabla_{m}F-(k-l)VF \\
&\quad +VF\left((n-k+1)\frac{\sigma_{k-1}}{\sigma_{k}}-(n-l+1)\frac{\sigma_{l-1}}{\sigma_{l}} \right) \\
&=VF\left((n-k+1)\frac{\sigma_{k-1}}{\sigma_{k}}-(n-l+1)\frac{\sigma_{l-1}}{\sigma_{l}}-(k-l) \right).
\end{align*} 
Then we know $F^{ij}\nabla^{2}_{ij}\widetilde{P}\geq 0$ from Lemma \ref{Th:IneqLem} and $V>0$.

\end{proof}

\section{Proof of Theorem \ref{Th:QuoHess}}
\label{Sec:ProofQuoHess}

From Lemma \ref{Th:P}, we know either \[P=K^{3}c_{1}^{2}+c_{2}^{2}-2Kc_{1}\quad \text{in }~\overline{\Omega} \] or \[ P<K^{3}c_{1}^{2}+c_{2}^{2}-2Kc_{1}\quad \text{in }~\Omega. \]

If the former occurs, namely $P$ is constant, calculations in the proof of Lemma \ref{Th:P} and the Newton-MacLaurin inequality imply $\nabla^{2}u+Kug=a(x)g$ for some function $a$. From the equation \[ \dfrac{\sigma_{k}(\nabla^{2}u+Kug)}{\sigma_{l}(\nabla^{2}u+Kug)}=\dfrac{\binom{n}{k}}{\binom{n}{l}}, \] we know $a(x)\equiv 1$. Hence $\nabla^{2}u+Kug=g$.

If the latter occurs, the Hopf Lemma implies
\begin{equation}\label{Eq:0<Pnu}
0<\nabla_{\nu}P(y)=2u_{\nu}u_{\nu\nu}+2Kuu_{\nu}-2u_{\nu}=2c_{2}(u_{\nu\nu}+K^{2}c_{1}-1)
\end{equation}
for any $y\in\partial\Omega$.

From Lemma \ref{Th:tildeP}, we know $F^{ij}\nabla^{2}_{ij}\widetilde{P}\geq 0$. We claim that $\widetilde{P}$ must be constant. Otherwise, by the strong maximum principle, there exists $y_{0}\in \partial\Omega$ such that $\widetilde{P}(x)<\widetilde{P}(y_{0})$ for any $x\in\Omega$. From Hopf Lemma, we know
\begin{align*}
\nabla_{\nu}\widetilde{P}(y_{0})>0.
\end{align*}
Since
\begin{align*}
\nabla_{\nu}\widetilde{P}&=-u_{\nu\nu}\Phi_{\nu}-u_{\nu}\Phi_{\nu\nu}+u_{\nu}V+uV_{\nu}+\Phi_{\nu} \\
&=-u_{\nu\nu}\Phi_{\nu}-u_{\nu}V+u_{\nu}V-Ku\Phi_{\nu}+\Phi_{\nu} \\
&=-\Phi_{\nu}(u_{\nu\nu}+Ku-1),
\end{align*}
we obtain 
\begin{equation}
0<\nabla_{\nu}\widetilde{P}(y_{0})=-\Phi_{\nu}(u_{\nu\nu}+K^{2}c_{1}-1).
\end{equation}
On the other hand, since the warping function $f(r)$ is positive and the domain $\Omega$ is star-shaped, we obtain
\begin{equation*}
\Phi_{\nu}=g(\nabla\Phi,\nu)=f(r)g(\frac{\partial}{\partial r},\nu)>0.
\end{equation*}
 This implies \[ u_{\nu\nu}+K^{2}c_{1}-1< 0. \] It contradicts \eqref{Eq:0<Pnu}. 

Now, since $\widetilde{P}$ is constant, calculations in the proof of Lemma \ref{Th:tildeP} and the Newton-MacLaurin inequality imply $\nabla^{2}u+Kug=a(x)g$ for some function $a$. As before, we obtain $\nabla^{2}u+Kug=g$.

At last, we show that $\nabla^{2}u+Kug=g$ implies $u$ is radially symmetric. The boundary condition $u\vert_{\partial\Omega}=Kc_{1}$ makes the discussion  different from \cite{Reilly80, Ciraolo-Vezzoni19}.

Since $u_{\nu}=c_{2}>0$ on $\partial\Omega$, we know there exists $p\in \Omega
$ such that $\displaystyle u(p)=\min_{\overline{\Omega}}u$. Let $\gamma: I\rightarrow \Omega$ be a maximal geodesic of arc-length parameter satisfying $\gamma(0)=p$. We consider $v(s):=u(\gamma(s))$, then
\begin{equation}\label{v''}
\left\{  
\begin{array}{lr} 
v''+Kv=1,\\  
v'(0)=0,\\
v(0)=u(p).
\end{array}
\right.
\end{equation}
Given two different such geodesics $\gamma_{1}$ and $\gamma_{2}$, by the existence and the uniqueness of the solution to the ODE, we have $u(\gamma_{1}(s))=u(\gamma_{2}(s))$,
hence $u$ is a radially symmetric function in the geodesic ball $B_{R}(p)$ where $R$ is chosen such that $\partial B_{R}(p)$ touches the $\partial \Omega,$ so $u=Kc_{1}$ on $\partial B_{R}(p)$. 

From the boundary condition of problem (\ref{Eq:hessian}), we know $\partial \Omega\subset\{u=Kc_{1}\}$. We need to show $\partial \Omega=\{u=Kc_{1}\}$. Here we use some auxiliary functions. Firstly we consider $\bar{u}:=u-Kc_{1}$, then $\nabla^{2}u+Kug=g$ implies 
\begin{align*}
\Delta \bar{u}+nK\bar{u}=\Delta u+nKu-nK^{2}c_{1}=n-nK^{2}c_{1}.
\end{align*}
 
Since $V$ is positive in $\Omega \subset M^{n}(K)$, we can set $w:=\dfrac{\bar{u}}{V}$. And by direct computation, we find that
\begin{align*}
w_{i}&=\frac{\bar{u}_{i}{V}-V_{i}\bar{u}}{V^{2}},\\
w_{ij}&=\frac{1}{V^{4}}\big((\bar{u}_{ij}V-V_{ij}\bar{u}+\bar{u}_{i}V_{j}-\bar{u}_{j}V_{i})V^{2}-2V V_{j}(\bar{u}_{i}V-V_{i}\bar{u})\big),\\
\Delta w&=\frac{1}{V}(\Delta \bar{u}+nK\bar{u})-\frac{2}{V}\langle \nabla V,\nabla w \rangle.
\end{align*}
Thus,
\begin{equation}\label{w eq}
\Delta w+\frac{2}{V}\langle \nabla V,\nabla w \rangle=\frac{n(1-K^{2}c_{1})}{V}.
\end{equation}

We claim that $K^{2}c_{1}<1$. Otherwise, the right hand side of \eqref{w eq} is non-positive. Thus by the strong maximum principle and $w\vert_{\partial{\Omega}}=0$, the function $w$ is either positive in $\Omega$ or $w\equiv0$ in $\Omega$. The latter cannot occur since it leads to $u\equiv Kc_{1}$ in $\Omega$ which contradicts $u_{\nu}\vert_{\partial{\Omega}}=c_{2}>0$. 
By the Hopf lemma, the former implies $w_ \nu<0$ on $\partial\Omega$ which contradicts $w_{\nu}=\frac{\bar{u}_{\nu}}{V}=\frac{u_{\nu}}{V}>0.$

Now, applying the strong maximum principle to \eqref{w eq}, we know that the function $w$ is negative in $\Omega$. Hence, $u<Kc_{1}$ in $\Omega$ which means $\partial \Omega=\{u=Kc_{1}\}$. As a result, $\partial \Omega=\partial B_{R}(p)$. Thus, we finish the proof.

\section{Rellich-Poho\v{z}aev type identity}
\label{Sec:R-Pidentity}

In the following, $\sigma_{k}=\sigma_{k}(b)$ for convenience. 
\begin{lemma}\label{Th:Poho}
Let $u\in C^{3}(\Omega)\cap C^{2}(\overline{\Omega})$ be a solution to the problem in Theorem \ref{Th:Hess}. Then
\begin{align*}
k\binom{n}{k}\int_{\Omega}uV \mathrm{d}\mu&=c_{1}K\int_{\partial\Omega}\sigma_{k}^{ij}u_{il}\Phi_{l}\nu_{j}\mathrm{d}S+\frac{c_{1}^{2}}{2}K^{3}\int_{\partial\Omega}\sigma_{k}^{ij}\Phi_{i}\nu_{j}\mathrm{d}S \\
&\quad +\frac{n-k+1}{2}K\int_{\Omega}\sigma_{k-1}u^{2}V\mathrm{d}\mu \\
&\quad -\frac{c_{2}^{2}}{2}\int_{\partial\Omega}\sigma_{k}^{li}\Phi_{l}\nu_{i}\mathrm{d}S+\frac{n-k+1}{2}\int_{\Omega}\sigma_{k-1}\abs{\nabla u}^{2}V\mathrm{d}\mu.
\end{align*}
\end{lemma}

\begin{proof}
From
\begin{align*}
k\sigma_{k}=\sigma_{k}^{ij}b_{ij} \qquad\text{ and }\qquad  \sigma_{k}=\binom{n}{k},
\end{align*}
we have
\begin{align*}
k\binom{n}{k}uV=\sigma_{k}^{ij}b_{ij}uV.
\end{align*}
By $b_{ij}=u_{ij}+Ku\delta_{ij}$ and $\sigma_{k}^{ij}\delta_{ij}=(n-k+1)\sigma_{k-1}$, we know
\begin{equation}\label{Eq:I}
\begin{aligned}
k\binom{n}{k}uV&=\sigma_{k}^{ij}(u_{ij}+Ku\delta_{ij})uV \\
&=\sigma_{k}^{ij}u_{ij}uV+(n-k+1)K\sigma_{k-1}u^{2}V.
\end{aligned}
\end{equation}

Using $\Phi_{ij}=V\delta_{ij}$ and $\nabla_{j}\sigma_{k}^{ij}=0$, we notice
\begin{equation}\label{Eq:I1}
\begin{aligned}
\sigma_{k}^{ij}u_{ij}uV&=\sigma_{k}^{ij}u_{il}u\Phi_{jl}\\
&=(\sigma_{k}^{ij}u_{il}u\Phi_{l})_{j}-\sigma_{k}^{ij}u_{ilj}u\Phi_{l}-\sigma_{k}^{ij}u_{il}u_{j}\Phi_{l}.
\end{aligned}
\end{equation}

Since $b_{ijk}=b_{ikj}$, $b_{ijk}=u_{ijk}+Ku_{k}\delta_{ij}$ and $\sigma_{k}=\binom{n}{k}$,
\begin{align*}
\sigma_{k}^{ij}u_{ilj}\Phi_{l}&=\sigma_{k}^{ij}b_{ilj}\Phi_{l}-K\sigma_{k}^{ij}u_{j}\delta_{il}\Phi_{l} \\
&=\nabla_{l}\sigma_{k}\Phi_{l}-K\sigma_{k}^{ij}u_{j}\Phi_{i} \\
&=-K\sigma_{k}^{ij}u_{j}\Phi_{i}.
        \end{align*}
Thus,
\begin{align*}
\sigma_{k}^{ij}u_{ilj}u\Phi_{l}&
=-\frac{1}{2}K\sigma_{k}^{ij}\Phi_{i}(u^{2})_{j} \\
&=-\frac{1}{2}K(\sigma_{k}^{ij}\Phi_{i}u^{2})_{j}+\frac{1}{2}K\sigma_{k}^{ij}\Phi_{ij}u^{2}.
\end{align*}
By \[ \sigma_{k}^{ij}\Phi_{ij}=V\sigma_{k}^{ij}\delta_{ij}=(n-k+1)\sigma_{k-1}V, \] we obtain
\begin{equation}\label{Eq:I11}
\sigma_{k}^{ij}u_{ilj}u\Phi_{l}=-\frac{1}{2}K(\sigma_{k}^{ij}\Phi_{i}u^{2})_{j}+\dfrac{n-k+1}{2}K\sigma_{k-1}u^{2}V.
\end{equation}
Since $u_{ij}$ and $b_{ij}$ can be diagonalized at the same time,
\begin{align*}
\sigma_{k}^{ij}u_{il}u_{j}\Phi_{l}=\sigma_{k}^{li}u_{ij}u_{j}\Phi_{l}.
\end{align*}
From
\begin{align*}
2\sigma_{k}^{li}u_{ij}u_{j}\Phi_{l}&=(\sigma_{k}^{li}\abs{\nabla u}^{2}\Phi_{l})_{i}-\sigma_{k}^{li}\abs{\nabla u}^{2}\Phi_{li} \\
&=(\sigma_{k}^{li}\abs{\nabla u}^{2}\Phi_{l})_{i}-(n-k+1)\sigma_{k-1}\abs{\nabla u}^{2}V,
\end{align*}
we obtain
\begin{equation}\label{Eq:I12}
\sigma_{k}^{ij}u_{il}u_{j}\Phi_{l}=\frac{1}{2}(\sigma_{k}^{li}\abs{\nabla u}^{2}\Phi_{l})_{i}-\frac{n-k+1}{2}\sigma_{k-1}\abs{\nabla u}^{2}V.
\end{equation}

Substituting \eqref{Eq:I11} and \eqref{Eq:I12} into \eqref{Eq:I1}, we get
\begin{align*}
\sigma_{k}^{ij}u_{ij}uV&=(\sigma_{k}^{ij}u_{il}u\Phi_{l})_{j}+\frac{1}{2}K(\sigma_{k}^{ij}\Phi_{i}u^{2})_{j}-\frac{n-k+1}{2}K\sigma_{k-1}u^{2}V \\
&\quad -\frac{1}{2}(\sigma_{k}^{li}\abs{\nabla u}^{2}\Phi_{l})_{i}+\frac{n-k+1}{2}\sigma_{k-1}\abs{\nabla u}^{2}V.
\end{align*}
Furthermore, it follows from \eqref{Eq:I} that
\begin{align*}
k\binom{n}{k}uV&=(\sigma_{k}^{ij}u_{il}u\Phi_{l})_{j}+\frac{1}{2}K(\sigma_{k}^{ij}\Phi_{i}u^{2})_{j} +\frac{n-k+1}{2}K\sigma_{k-1}u^{2}V \\
&\quad -\frac{1}{2}(\sigma_{k}^{li}\abs{\nabla u}^{2}\Phi_{l})_{i}+\frac{n-k+1}{2}\sigma_{k-1}\abs{\nabla u}^{2}V.
\end{align*}

Integrating and using the divergence theorem, we obtain
\begin{align*}
k\binom{n}{k}\int_{\Omega}uV \mathrm{d}\mu&=\int_{\partial\Omega}\sigma_{k}^{ij}u_{il}u\Phi_{l}\nu_{j}\mathrm{d}S+\frac{1}{2}K\int_{\partial\Omega}\sigma_{k}^{ij}\Phi_{i}u^{2}\nu_{j}\mathrm{d}S \\
&\quad +\frac{n-k+1}{2}K\int_{\Omega}\sigma_{k-1}u^{2}V\mathrm{d}\mu \\
&\quad -\frac{1}{2}\int_{\partial\Omega}\sigma_{k}^{li}\abs{\nabla u}^{2}\Phi_{l}\nu_{i}\mathrm{d}S+\frac{n-k+1}{2}\int_{\Omega}\sigma_{k-1}\abs{\nabla u}^{2}V\mathrm{d}\mu.
\end{align*}

By substituting the boundary conditions  $u\vert_{\partial\Omega}=Kc_{1}$ and $\abs{\nabla u}^{2}\vert_{\partial\Omega}=c_{2}^{2}$, we obtain the desired identity.
\end{proof}

In order to deal with the integral terms on the boundary $\partial\Omega$, we need to show the following lemma. 

\begin{lemma}\label{Th:bdryterm}
Under the same assumption of Lemma \ref{Th:Poho}, the following identities hold
\begin{itemize}
\item[i)] $\displaystyle \int_{\partial\Omega}\sigma_{k}^{ij}\Phi_{i}\nu_{j} \,\mathrm{d}S=(n-k+1)\int_{\Omega}\sigma_{k-1}V \,\mathrm{d}\mu,$
\item[ii)] $\displaystyle \int_{\partial\Omega}\sigma_{k}^{ij}u_{il}\Phi_{l}\nu_{j} \,\mathrm{d}S=\int_{\Omega} \left(k\binom{n}{k} -K^{2}c_{1}(n-k+1)\sigma_{k-1}\right)V\, \mathrm{d}\mu$.
\end{itemize}
\end{lemma}

\begin{proof}
By the divergence theorem, $\Phi_{ij}=V\delta_{ij}$ and $\sigma_{k}^{ij}\delta_{ij}=(n-k+1)\sigma_{k-1}$, we have
\begin{equation*}
\int_{\partial\Omega}\sigma_{k}^{ij}\Phi_{i}\nu_{j} \,\mathrm{d}S=\int_{\Omega}(\sigma_{k}^{ij}\Phi_{i})_{j} \,\mathrm{d}\mu=(n-k+1)\int_{\Omega}\sigma_{k-1}V \,\mathrm{d}\mu.
\end{equation*}

Similarly, using the conditions
$u\vert_{\partial\Omega}=Kc_{1}\text{ and }\sigma_{k}=\binom{n}{k}$, 
we get
\begin{align*}
\int_{\partial\Omega}\sigma_{k}^{ij}u_{il}\Phi_{l}\nu_{j} \,\mathrm{d}S&=\int_{\partial\Omega}(\sigma_{k}^{ij}b_{il}\Phi_{l}\nu_{j}-K\sigma^{ij}_{k}u\delta_{il}\Phi_{l}\nu_{j}) \,\mathrm{d}S \\
&=\int_{\Omega}\sigma_{k}^{ij}b_{il}\Phi_{lj}\,\mathrm{d}\mu -K^{2}c_{1}\int_{\partial\Omega}\sigma_{k}^{ij}\Phi_{i}\nu_{j} \,\mathrm{d}S \\
&=\int_{\Omega}(\sigma_{k}^{ij}b_{ij}V-K^{2}c_{1}\sigma_{k}^{ij}\Phi_{ij})\,\mathrm{d}\mu.
\end{align*}
Thus, from the equalities
\begin{equation*}
\sigma_k^{ij}b_{ij}=k\sigma_{k}=k\binom{n}{k}
\end{equation*}
and
 \begin{equation*}
\sigma_k^{ij}\Phi_{ij}=\sigma_k^{ij}\delta_{ij}V=(n-k+1)\sigma_{k-1}V,
\end{equation*}
the second identity follows.

\end{proof}

Combining Lemma \ref{Th:Poho} and Lemma \ref{Th:bdryterm}, we obtain the following lemma.

\begin{lemma}\label{Th:intP}
Let $u\in C^{3}(\Omega)\cap C^{2}(\overline{\Omega})$ be a solution to the problem in Theorem \ref{Th:Hess}. Then
\begin{equation}\label{Eq:intuV}
\begin{aligned}
&\binom{n}{k-1}\int_{\Omega}(u-Kc_{1})V \,\mathrm{d}\mu \\
&\qquad=\frac{1}{2}\int_{\Omega}(\abs{\nabla u}^{2}+Ku^{2}-K^{3}c_{1}^{2}-c_{2}^{2})\sigma_{k-1}V \,\mathrm{d}\mu.
\end{aligned}
\end{equation}
\end{lemma}

\begin{proof}
Substituting the identities in Lemma \ref{Th:bdryterm} into Lemma \ref{Th:Poho}, we have
\begin{align*}
k\binom{n}{k}\int_{\Omega}uV \mathrm{d}\mu&=k\binom{n}{k}c_{1}K\int_{\Omega}V\,\mathrm{d}\mu-(n-k+1)c_{1}^{2}K^{3}\int_{\Omega}\sigma_{k-1}V \,\mathrm{d}\mu \\
&\quad +\frac{n-k+1}{2}c_{1}^{2}K^{3}\int_{\Omega}\sigma_{k-1}V \,\mathrm{d}\mu +\frac{n-k+1}{2}K\int_{\Omega}\sigma_{k-1}u^{2}V \,\mathrm{d}\mu \\
&\quad -\frac{n-k+1}{2}c_{2}^{2}\int_{\Omega}\sigma_{k-1}V \,\mathrm{d}\mu +\frac{n-k+1}{2}\int_{\Omega}\sigma_{k-1}\abs{\nabla u}^{2}V \,\mathrm{d}\mu.
\end{align*}
The above equality is equivalent to
\begin{align*}
&k\binom{n}{k}\int_{\Omega}(u-Kc_{1})V \,\mathrm{d}\mu \\
&\qquad=\frac{n-k+1}{2}\int_{\Omega}(\abs{\nabla u}^{2}+Ku^{2}-K^{3}c_{1}^{2}-c_{2}^{2})\sigma_{k-1}V \,\mathrm{d}\mu.
\end{align*}
Noticing $\binom{n}{k-1}=\frac{k}{n-k+1}\binom{n}{k}$, we finish the proof.
\end{proof}

\section{Proof of Theorem \ref{Th:Hess}}
\label{Sec:ProofkHess}

First, we show that $u< Kc_{1}$ in $\Omega$.

Since
\begin{align*}
\sigma_{k}^{ij}u_{ij}=\sigma_{k}^{ij}(b_{ij}-Ku\delta_{ij})=k\sigma_{k}-(n-k+1)Ku\sigma_{k-1},
\end{align*}
we obtain
\begin{align*}
\sigma_{k}^{ij}u_{ij}+(n-k+1)Ku\sigma_{k-1}=k\binom{n}{k}>0.
\end{align*}
For the case $K=0$, by the strong maximum principle, the above inequality shows $u(x)<u\vert_{\partial\Omega}=Kc_{1}$ for all $x\in\Omega$.

For the case $K\neq 0$,  we consider $\bar{u}:=u-Kc_{1}$. Then
\begin{align*}
\Delta \bar{u}+nK\bar{u}=\Delta u+nKu-nK^{2}c_{1}=\sigma_{1}(b)-nK^{2}c_{1}.
\end{align*}

The Newton-MacLaurin inequality implies
\begin{equation*}
\frac{\sigma_{1}(b)}{n}\geq \left(\frac{\sigma_{k}(b)}{\binom{n}{k}}\right)^{\frac{1}{k}}=1.
\end{equation*}

Thus, we have
\begin{equation*}
\Delta \bar{u}+nK\bar{u}\geq n(1-K^{2}c_{1})\geq 0,
\end{equation*}
where the last inequality is from our assumption $K^{2}c_{1}\leq 1$.
 
Since $V$ is a positive function in $\Omega \subset M^{n}(K)$, we can set $w:=\dfrac{\bar{u}}{V}$. By direct computation we find that
\begin{align*}
\Delta w+\frac{2}{V}\langle \nabla V,\nabla w \rangle&=\frac{1}{V}(\Delta \bar{u}+nK\bar{u}).
\end{align*}
Thus,
\begin{equation*}
\Delta w+\frac{2}{V}\langle \nabla V,\nabla w \rangle\geq 0.
\end{equation*}

If $w$ is constant, $\bar{u}\vert_{\partial\Omega}=0$ implies $u=Kc_{1}$ which contradicts $u_{\nu}=c_{2}>0$. Therefore, the strong maximum principle implies $w<0$ in $\Omega$. Hence, $u<Kc_{1}$ in $\Omega$.

Now, from Lemma \ref{Th:P}, we know that either \[P=K^{3}c_{1}^{2}+c_{2}^{2}-2Kc_{1}\quad \text{in }~\overline{\Omega}, \] or \[ P<K^{3}c_{1}^{2}+c_{2}^{2}-2Kc_{1}\quad \text{in }~\Omega. \] And we only need to show the latter can not occur. If $P=\abs{\nabla u}^{2}+Ku^{2}-2u<K^{3}c_{1}^{2}+c_{2}^{2}-2Kc_{1}$ in $\Omega$, using $\sigma_{k-1}>0$ and $V>0$ in $\Omega$, we have
\begin{equation}\label{k=1}
\int_{\Omega}(\abs{\nabla u}^{2}+Ku^{2}-K^{3}c_{1}^{2}-c_{2}^{2})\sigma_{k-1}V \,\mathrm{d}\mu<2\int_{\Omega}(u-Kc_{1})\sigma_{k-1}V \,\mathrm{d}\mu.
\end{equation}

For the case $k=1$, the above inequality contradicts the equality \eqref{Eq:intuV} in Lemma \ref{Th:intP}. 

For the case $2\leq k\leq n$, the Newton-MacLaurin inequality implies 
\begin{equation*}
\frac{\sigma_{k-1}(b)}{\binom{n}{k-1}}\geq \left(\frac{\sigma_{k}(b)}{\binom{n}{k}}\right)^{\frac{k-1}{k}}=1.
\end{equation*}

Since $u< Kc_{1}$ in $\Omega$, we obtain
\begin{equation*}
\int_{\Omega}(\abs{\nabla u}^{2}+Ku^{2}-K^{3}c_{1}^{2}-c_{2}^{2})\sigma_{k-1}V \,\mathrm{d}\mu<2\binom{n}{k-1}\int_{\Omega}(u-Kc_{1})V \,\mathrm{d}\mu.
\end{equation*}
This also contradicts the equality \eqref{Eq:intuV} in Lemma \ref{Th:intP}. 

Thus $P$ must be constant in $\overline{\Omega}$. The rest of the proof is the same as in Section \ref{Sec:ProofQuoHess}.

{\bf Note added in proof.}
Just before we submitted the present paper, we were informed of the paper written by Z. Gao, X. Jia, and D. Zhang (arXiv: 2209.06268). They consider Serrin-type overdetermined problems for the Hessian quotient equations with zero Dirichlet condition in $\mathbb{R}^n$ and $\mathbb{H}^n$ and establish corresponding Rellich-Poho\v{z}aev type identity. Our results and methods are different from theirs. We consider overdetermined problems with constant Dirichlet boundary conditions in space forms, i.e., $\mathbb{R}^{n}$, $\mathbb{H}^{n}$ and $\mathbb{S}^{n}_{+}$. And we prove Lemma \ref{Th:binGammak}, so we do not additionally assume $\sigma_{l}(\nabla^{2}u+Kug)>0$ in $\overline{\Omega}$ as in their results. We also introduce the auxiliary function $\widetilde{P}$ for the Hessian quotient equations in space forms. As a result, the proof of Theorem \ref{Th:QuoHess} does not rely on Rellich-Poho\v{z}aev type identity. Recently, inspired by the work presented in the paper arXiv: 2209.06268, we can establish Rellich-Poho\v{z}aev type identity for problem \eqref{Eq:hessian}. Consequently, Theorem \ref{Th:QuoHess} holds if the star-shapedness condition is replaced with $K^2 c_1 \leq 1$.

\begin{acknow}
The authors would like to thank Chao Qian, Chao Xia and Jiabin Yin for their helpful conversations on this work. We are also grateful to the anonymous reviewer for helpful comments.

The first named author is partially supported by Natural Science Basic Research Program of Shaanxi (Program No.2022JQ-065), The Youth Innovation Team of Shaanxi Universities and Shaanxi Fundamental Science Research Project for Mathematics and Physics (Grant No.22JSZ012). The second and the third named authors are partially supported by National Natural Science Foundation of China (Grants No.~11831005 and No.~12061131014).
\end{acknow}




\end{document}